\documentclass{amsart}
\usepackage{mathrsfs,amssymb,multirow,setspace,color,amsmath}
\usepackage[a4paper, total={7in, 9in}]{geometry}
\theoremstyle{plain}
\usepackage{graphicx}
\usepackage{mathtools}

\newtheorem{theorem}{Theorem}[section]
\newtheorem{lemma}[theorem]{Lemma}
\newtheorem{proposition}[theorem]{Proposition}
\newtheorem{corollary}[theorem]{Corollary}

\theoremstyle{definition}

\newtheorem{example}[theorem]{Example}

\theoremstyle{remark}
\newtheorem{remark}[theorem]{Remark}

\input xy
\xyoption{all}

\def\c{{\mathcal{P}_E(G)}}
\def\dd{\mathcal{P}_E^{**}(G)}
\def\po{\mathcal{P}(G)}
\def\ppo{\mathcal{P}^{**}(G)}
\def\t{\mathcal{T}(G)}
\def\del{\Delta(G)}

\def\m{\mathcal{M}(G)}

\begin{document}
\title[Line Graph of Power Graphs of Finite Groups]{On Finite groups whose power graphs are line graphs}
%\maketitle
%\section*{Introduction}
%******************
%\title[Certain properties of the enhanced power graph associated with a finite group]{Certain properties of the enhanced power graph associated with a finite group}

%\author[R. P. Panda]{Ramesh Prasad Panda}

%\author[S. Dalal]{Sandeep Dalal}

%\author[J. Kumar]{Jitender Kumar}
\author[Parveen, Jitender Kumar]{Parveen, $\text{JITENDER KUMAR}^{^*}$}
\address{$\text{}^1$Department of Mathematics, Birla Institute of Technology and Science Pilani, Pilani-333031, India}
\email{p.parveenkumar144@gmail.com,jitenderarora09@gmail.com}
%\address{School of Science, Xi’an Shiyou University, Xi’an 710065, China}
%\email{p.parveenkumar144@gmail.com,jitenderarora09@gmail.com,sidharth$\_$0903@hotmail.com,xuanlma@mail.bnu.edu.cn}

%\date{...}
\begin{abstract}
S. Bera (Line graph characterization of power graphs of finite nilpotent groups, \textit{Communication in Algebra}, 50(11), 4652-4668, 2022) characterized finite nilpotent groups whose power graphs and proper power graphs are line graphs. In this paper, we extend the results of above mentioned paper to arbitrary finite groups. Also, we correct the corresponding result  of the proper power graphs of dihedral groups. Moreover, we classify all the finite groups whose enhanced power graphs are line graphs. We classify all the finite nilpotent groups (except non-abelian $2$-groups) whose proper enhanced power graphs are line graphs of some graphs. Finally, we determine all the finite groups whose power graphs, proper power graphs, enhanced power graphs and proper enhanced power graphs are the complement of line graphs, respectively. 
\end{abstract}

\keywords{Power graph, enhanced power graph, line graph, nilpotent groups. \\ * Corresponding Author}

\maketitle

\section{Historical background}
The study of graphs associated to algebraic structures is a large research area and one of the important topic of research of algebraic graph theory. Such study provides an interplay between algebra and graph theory. The study of graphs associated to groups have been studied extensively because they have valuable applications and related to automata theory  
(see \cite{kelarev2003graph, kelarev2004labelled,a.kelarev2009cayley}). Graphs associated to groups, viz: Cayley graphs, power graphs, commuting graphs, enhanced power graphs, prime graphs, intersection graphs etc., have been studied by various researchers. Kelarev \emph{et. al} \cite{a.kelarev2000groups} introduced the notion of power graphs. The \emph{power graph} $\mathcal{P}(G)$ of a group $G$ is a simple undirected graph with vertex set $G$ such that two vertices $a$ and $b$ are adjacent if one is a power of the other or equivalently: either $a \in  \langle b\rangle$ or $b \in \langle a\rangle$.  Cameron \cite{a.Cameron2010} proved that two finite groups which have isomorphic power graphs have the same number of elements of each order. Further, Cameron \emph{et. al} \cite{a.Cameron2011} showed that two finite abelian groups are isomorphic if and only if their power graphs are isomorphic. A graph is said to be $\Gamma$-free if it has no induced subgraph isomorphic to $\Gamma$.  Doostabadi \emph{et. al} \cite{a.doostabadiforbidden} characterized all the finite groups whose power graphs are $K_{1,3}$-free, $K_{1,4}$-free or $C_4$-free. Power graphs of groups with certain forbidden subgraphs such as split, threshold, chordal and cograph have been investigated in \cite{a.MannaForbidden2021}.  For more results on power graphs of groups, we refer the reader to \cite{a.kelarev2000groups,a.powergraphsurvey} and references therein. The dominating vertices of a graph are the one which are adjacent to all other vertices of the graph. The study of connectedness of the graphs obtained by deleting dominating vertices becomes important and interesting. Proper power graph $\ppo$ (is the graph obtained from $\po$ after deleting its dominating vertices) of a group $G$ is also studied in the literature. The connectivity of proper power graphs for certain groups including nilpotent groups have been studied in \cite{a.doostabadi2015connectivity}. Further, Cameron and Jafari \cite{a.cameron2020connectivity} discussed the connectivity of  proper power graph of an arbitrary finite group and characterize all groups whose power graphs have finite independence number. They showed that a group whose proper power graph is connected must be either a torsion group or a torsion-free group. Also, they classify those groups whose power graph is dominatable.

In order to study how close the power graph is to the commuting graph of a finite group $G$, Aalipour \emph{et. al} \cite{a.Cameron2016} introduced the \emph{enhanced power graph}.  The \emph{enhanced power graph} $\c$ of a group $G$ is a simple undirected graph with vertex set $G$ and two vertices $x$ and $y$ are adjacent if $x,y\in \langle z \rangle$ for some $z\in G$. Equivalently, two vertices $x$ and $y$ are adjacent in $\c$ if and only if $\langle x,y\rangle$ is a cyclic subgroup of $G$. Note that the power graph $\po$ is a spanning subgraph of the enhanced power graph $\c$. Bera and Bhuniya \cite{a.Bera2017} studied the interconnection between algebraic properties of the group $G$ and graph theoretic properties of its enhanced power graph $\c$. They proved that the enhanced power graph $\c$ is Eulerian if and only if $G$ is of odd order. Also, they characterized the abelian groups and non-abelian $p$-groups having dominatable enhanced power graphs. Together with certain graph theoretic invariants such as minimum degree, independence number, strong metric dimension, matching number etc., Panda \emph{et al.} \cite{a.panda2021enhanced} studied perfectness of the enhanced power graphs of certain groups including finite abelian $p$-groups. Zahirovi\'{\rm c} \emph{et al.} \cite{a.zahirovic2020study} proved that two finite abelian groups are isomorphic if and only if their enhanced power graphs are isomorphic. Also, they supplied a characterization of finite nilpotent groups whose enhanced power graphs are perfect. Enhanced power graphs of groups with certain forbidden subgraphs such as split, threshold, chordal and cograph have been investigated in \cite{a.ma2021forbidden}. A detailed list of results and open problems related to enhanced power graph of a group can be found in \cite{a.masurvey2022}. Moreover, the proper enhanced power graph (is the graph obtained from $\c$ after deleting its dominating vertices) is also studied in the literature. Bera \emph{et al.} \cite{a.bera2022dominating} classified all nilpotent groups whose proper enhanced power graph is connected and calculated their diameter. Moreover, they determined the domination number of proper enhanced power graphs of finite nilpotent groups. Bera \emph{et. al} \cite{a.bera2021connectivity} 
 computed the number of connected components of proper enhanced power graph. Moreover, they studied the connectivity of proper enhanced power graphs of certain non-abelian groups.

The \emph{line graph} $L(\Gamma)$ of the graph $\Gamma$ is a graph whose vertex set is all the edges of  $\Gamma$ and two vertices of $L(\Gamma)$ are adjacent if they are incident in $\Gamma$.
 An example of $L(\Gamma)$ of the graph $\Gamma$ is shown  in Figure \ref{arbitary line graph}. Line graphs are described by nine forbidden subgraphs (cf. Theorem \ref{induced lemma}). Recently, Bera \cite{a.bera2022} classified all the finite nilpotent groups whose power graphs and proper power graphs are line graphs. Motivated by the work of \cite{a.bera2022}, in the present paper, we intend to study the line graph of certain power graphs associated to finite groups, viz: power graph, proper power graph, enhanced power graph, proper enhanced power graph. In order to extend the results of \cite{a.bera2022}, we study the following problems.

\begin{itemize}
    \item Classification of finite groups $G$ such that $\del \in \{\po, \ppo, \c, \dd \} $ is a line graph.
     \item Classification of finite groups $G$ such that $\del \in \{\po, \ppo, \c, \dd \} $ is the complement of a line graph.
\end{itemize}

The main results of this manuscript are stated in Section 3. 

\section{Preliminaries}
A graph $\Gamma$ consists of a vertex set $V(\Gamma)$ and an edge set $E(\Gamma)$, where $E(\Gamma)$ is an unordered subset of $V(\Gamma) \times V(\Gamma)$. If $\{u,v\} \in E(\Gamma)$, then we say $u$ and $v$ are \emph{adjacent} and we write as $u \sim v$.  Otherwise, we write $u \nsim v$. If $\{u,v\} \in E(\Gamma)$ then the vertices $u$ and $v$ are called \emph{endpoints} of the edge $\{u,v\}$. Two edges $e_1$ and $e_2$ are said to be \emph{incident} if they have a common endpoint. An edge $\{u,v\}$ is called a \emph{loop} if $u=v$. A graph without loops or repeated edges is called \emph{simple graph}. Throughout this paper, we are considering only finite simple graphs. Let $\Gamma$ be a graph. The set of all vertices adjacent to a vertex $u$ is called \emph{neighbours} of $u$ in $\Gamma$, and it is denoted by $N(u)$, $N[u]= N(u)\cup \{u\}$. The \emph{degree} of a vertex $u$ in a graph $\Gamma$ is the cardinality of $N(u)$ in $\Gamma$. A \emph{subgraph} of a graph $\Gamma$ is a graph $\Gamma'$ such that $V(\Gamma')\subseteq V(\Gamma)$ and $E(\Gamma')\subseteq E(\Gamma)$. If $V(\Gamma')= V(\Gamma)$, we call $\Gamma '$ a \emph{spanning subgraph} of $\Gamma$. A  subgraph $\Gamma'$ of $\Gamma$ is an \emph{induced subgraph} by a set $X$ if $V(\Gamma')=X$ and two vertices of $\Gamma'$ are adjacent if they are adjacent in $\Gamma$. A vertex $u$ is said to be a \emph{dominating vertex} of a graph $\Gamma$ if $u$ is adjacent to all other vertices $\Gamma$, and it is denoted by $\mathrm{Dom}(\Gamma)$. A graph $\Gamma$ is called \emph{complete}  if every vertex of $\Gamma$ is a dominating vertex, and the complete graph on $n$ vertices is denoted by $K_n$.  A graph $\Gamma$ is said to be \emph{bipartite} if $V(\Gamma)$ can be partitioned into two subsets such that no two vertices in the same partition subset are adjacent. A \emph{complete bipartite} graph is a bipartite graph such that every vertex in one part is adjacent to all the vertices of the other part. A \emph{complete bipartite graph} with partition size $m$ and $n$ is denoted by $K_{m, n}$. A complete bipartite graph $K_{1,n}$ is called a \emph{star graph}.  
  \begin{figure}[ht]
    \centering
    \includegraphics[scale=.9]{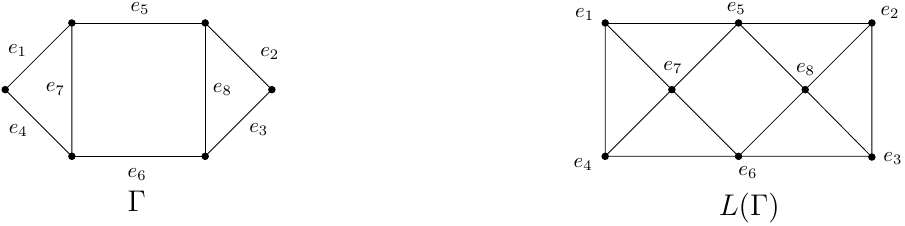}
    \caption{A graph $\Gamma$ and its line graph $L(\Gamma)$}
    \label{arbitary line graph}
\end{figure}
The \emph{complement} of a graph $\Gamma$ is a graph $\overline{\Gamma}$ such that $V(\overline{\Gamma})= V(\Gamma)$ and two vertices are adjacent in $\overline{\Gamma}$ if and only if they are not adjacent in $\Gamma$.
A \emph{path} of length $r$ from $u$ to $v$ in a graph is a sequence of $r+1$ distinct vertices starting with $u$ and ending with $v$ such that consecutive vertices are adjacent. If there is a path between any two vertices of a graph, then $\Gamma$ is \emph{connected}, otherwise \emph{disconnected}. A maximal connected subgraph $\Gamma'$, of a graph $\Gamma$, is called \emph{component}. A \emph{path graph} is a connected graph having at least $2$ vertices and it has two (terminal) vertices that have degree $1$, while all other vertices have degree $2$. We denote by $P_n$, a path graph on $n$ vertices. Let $\Gamma_1,\ldots , \Gamma_m$ be $m$ graphs such that $V(\Gamma_i)\cap V(\Gamma_j)= \phi$, for $i\neq j$. Then $\Gamma =\Gamma_1 \cup \cdots \cup \Gamma_m$ is a graph with vertex set  $V(\Gamma_1) \cup \cdots \cup V(\Gamma_m)$ and edge set $E(\Gamma_1) \cup \cdots \cup E(\Gamma_m)$. Let $\Gamma_1$ and $\Gamma _2$ be two graphs with disjoint vertex set, the \emph{join} $\Gamma_1 \vee \Gamma_2$ of $\Gamma_1$ and $\Gamma_2$ is the graph obtained from the union of $\Gamma_1$ and $\Gamma_2$ by adding new edges from each vertex of $\Gamma_1$ to every vertex of $\Gamma_2$.  Two graphs $\Gamma_1$ and $\Gamma _2$ are \emph{isomorphic} if there is a bijection, $f$ from $V(\Gamma _1)$ to $V(\Gamma _2)$ such that if $u\sim v$ in $\Gamma _1$ if and only if $f(u)\sim f(v)$ in $\Gamma_2$.

Next lemma is very important relation for characterization of line graph.
\begin{lemma}{\rm \cite{a.beineke1970}}{\label{induced lemma}}
    Let $\Gamma$ be a graph. Then $\Gamma$ is the line graph of some graph if and only if none of the nine graphs in $\mathrm{Figure \; \ref{fig line graph}}$  is an induced subgraph of $\Gamma$.
\end{lemma}
   \begin{figure}[ht]
    \centering
    \includegraphics[scale=.9]{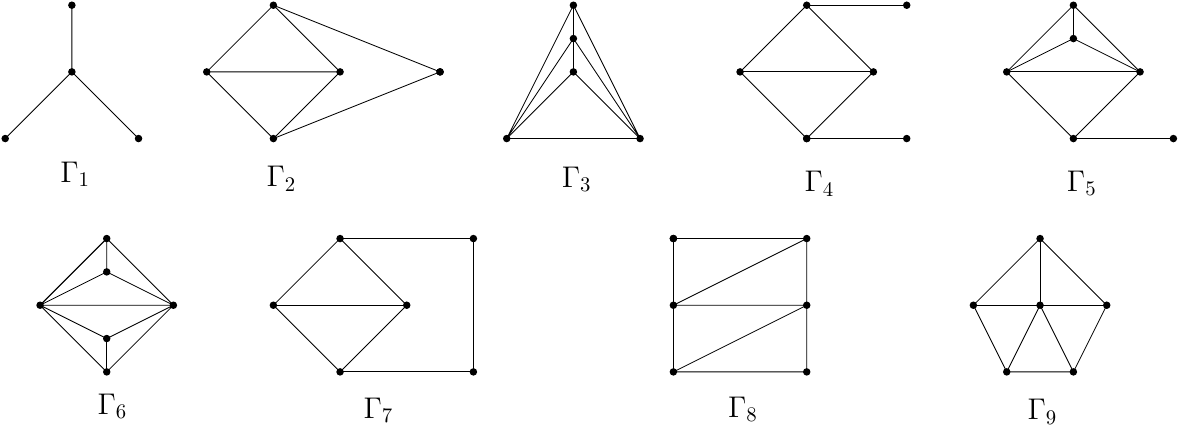}
    \caption{Forbidden induced subgraphs of line graphs}
    \label{fig line graph}
\end{figure}
For characterization of complement of line graph we use the following lemma.
\begin{lemma}{\rm \cite[Theorem 3.1]{a.barati2021}}{\label{complement induced lemma}}
A graph $\Gamma$ is the complement of a line graph if and only if none of the nine graphs $\overline{\Gamma_i}$ of $\mathrm{Figure \; \ref{fig complement_line_graph}}$, is an induced subgraph of $\Gamma$.
   \begin{figure}[ht]
    \centering
    \includegraphics[scale=.9]{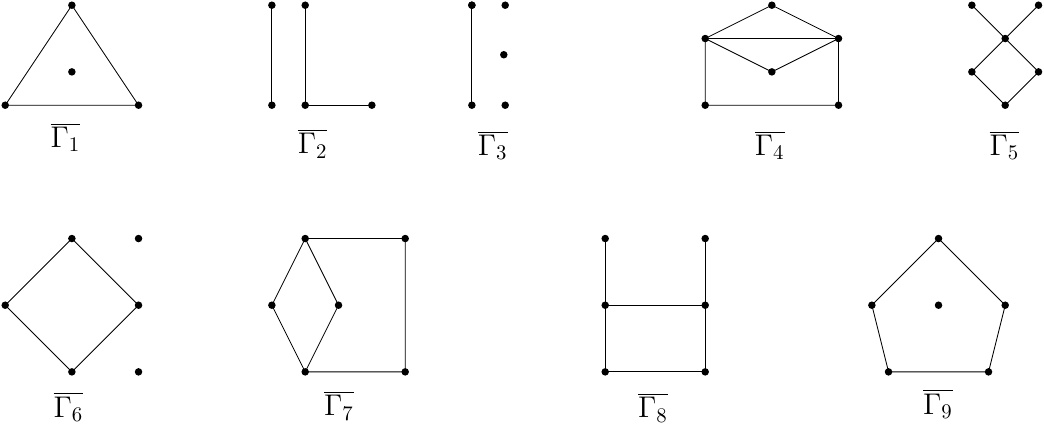}
    \caption{Forbidden induced subgraphs of Complement of line graphs}
    \label{fig complement_line_graph}
\end{figure}
\end{lemma}
We shall use $\Gamma_i$(or $\overline{\Gamma_i})$, for $1\leq i \leq 9$, explicitly in this paper without referring to it. \\
Let $G$ be a finite group.  We write $o(x)$ by the order of an element $x$ in $G$. For a positive integer $n$, $\phi(n)$ denotes the Euler's totient function of $n$. A cyclic subgroup of the group $G$ is called a \emph{maximal cyclic subgroup} if it is not contained in any cyclic subgroup of $G$ other than itself. We denote $\m$ by the set of all maximal cyclic subgroups of $G$ and $\mathcal{M}^{(p)}(G)=\{M\in \m : M \text{ is a }p\text{-group}\}$. Note that $|\m|=1$ if and only if $G$ is a cyclic group. The intersection of all the maximal cyclic subgroups of $G$ is denoted by $\t$.  
The following remark is useful in the sequel.
\begin{remark}{\label{remark maximal}}
Let $G$ be a finite group. Then $G= \bigcup\limits_{M\in \m} M$ and the generators of a maximal cyclic subgroup does not belong to any other maximal cyclic subgroup of $G$.
\end{remark}
\begin{remark}{\label{remark}} 
(i) In the enhanced power graph $\c$, $x\sim y$ if and only if $x,y\in M$ for some $M\in \m$. Consequently, $\mathrm{Dom}(\c)= \t$.\\
(ii) In the power graph $\po$, $x\sim y$ if and only if  $x, y \in M$ for some $M\in \m$ and either $o(x)\vert o(y)$ or $o(y)\vert o(x)$.\\
(iii) Let $M\in \m$ such that $M=\langle m \rangle$. Then $N[m]=M$ in $\c$ as well as in $\po$. Moreover, if $G$ is non-cyclic then $m\notin \t$ and so $m\in V(\dd)$.\\
(iv) Let $M, M'\in\m $ such that $M=\langle m \rangle$ and $M'=\langle m'\rangle$. Then $m\nsim m'$ in $\c$ and so $m\nsim m'$ in  $\Gamma (G)\in \{  \po, \ppo, \dd\}$.

\end{remark}
\begin{lemma}{\rm \cite[Lemma 2.7]{a.kumar2022complement}}{\label{M(G) >2}}
    If $G$ is a finite group, then $| \m |\neq 2$.
\end{lemma}

\begin{theorem}{\rm \cite{b.dummit1991abstract}}{\label{nilpotent}}
 Let $G$ be a finite group. Then the following statements are equivalent:
 \begin{enumerate}
     \item[(i)] $G$ is a nilpotent group.
     \item[(ii)] Every Sylow subgroup of $G$ is normal.
    \item[(iii)] $G$ is the direct product of its Sylow subgroups.
    \item[(iv)] For $x,y\in G, \  x$ and $y$ commute whenever $o(x)$ and $o(y)$ are relatively primes.
 \end{enumerate}
 \end{theorem}
\begin{lemma}{\rm \cite[Lemma 2.11]{a.chattopadhyay2021minimal}}{\label{on maximal connectivity}}
Any maximal cyclic subgroup of a finite nilpotent group $G=P_1\times P_2\times \cdots \times P_r$ is of the form $M_1\times M_2\times \cdots \times M_r $, where $M_i$ is a maximal cyclic subgroup of $P_i,  \ (1 \leq i \leq r)$.
\end{lemma}
  For $n \geq 2$, the \emph{generalized quaternion group} $Q_{4n}$ is defined in terms of generators and
relations as $Q_{4n} = \langle a, b  :  a^{2n}  = e'' , a^n= b^2, ab = ba^{-1} \rangle.$
Let $G_1$ be a nilpotent group having no Sylow subgroups that are cyclic or generalized quaternion. Suppose $e, e^{\prime \prime}$ are the identities of $G_1$ and $Q_{2^k}$ respectively. Let $\mathcal{D}_1= \left\{\left(e, x, e^{\prime \prime}\right): x \in \mathbb{Z}_n\right\}$ and $\mathcal{D}_2=\left\{(e, x, y): x \in \mathbb{Z}_n, y \in Q_{2^k}\right.$ and $\left.o(y)=2\right\}$.
\begin{theorem}{\rm \cite[Corollary 4.2]{a.bera2022dominating}}{\label{Dominating enhanced power graph}}
    Let $G$ be a finite nilpotent group. Then
$$
\operatorname{Dom}\left(\c \right)= \begin{cases}\{e\} & \text { if } G=G_1, \\ \left\{(e, x): x \in \mathbb{Z}_n\right\} & \text { if } G=G_1 \times \mathbb{Z}_n \text { and } \operatorname{gcd}\left(\left|G_1\right|, n\right)=1, \\ \left\{\left(e, e^{\prime \prime}\right),(e, y)\right\} & \text { if } G=G_1 \times Q_{2^k} \text { and } \operatorname{gcd}\left(\left|G_1\right|, 2\right)=1, \\ \mathcal{D}_1 \cup \mathcal{D}_2 & \text { if } G=G_1 \times \mathbb{Z}_n \times Q_{2^k}  \text { and } \operatorname{gcd}\left(\left|G_1\right|, n\right)=\operatorname{gcd}\left(\left|G_1\right|, 2\right) =\operatorname{gcd}(n, 2)=1 .\end{cases}
$$
\end{theorem}

\begin{theorem}{\rm \cite[Theorem 2.4]{a.Bera2017}}{\label{complete}}
The enhanced power graph $\c$   of the group $G$  is complete if and only if $G$ is cyclic.
\end{theorem}
\begin{theorem}{\rm \cite[Theorem 2.12]{a.chakrabarty2009undirected}}{\label{complete power graph}}
For a finite group $G$, the power graph $\po$  is complete if and only if $G$ is a cyclic group of order $1$ or $p^m$, for some prime $p$ and for some $m\in \mathbb{N}$.
\end{theorem}

\begin{theorem}{\rm \cite[Theorem 4]{a.cameron2020connectivity}}{\label{dominating power graph}}
    Let $G$ be a finite group. Suppose that $x\in G$ has the property that for all $y\in G$, either $x$ is a power of $y$ or vice versa. Then one of the following holds:
    \begin{itemize}
        \item[(i)] $x=e$;
        \item[(ii)] $G$ is cyclic and $x$ is a generator;
        \item[(iii)] $G$ is a cyclic $p$-group for some prime $p$ and $x$ is arbitrary;
        \item[(iv)] $G$ is a generalized quaternion group and $x$ is order $2$.
    \end{itemize}
\end{theorem}

\section{Main results}
The  main results of the manuscript are stated in this section.  In the following three theorems, the results of \cite{a.bera2022} are extended from nilpotent groups to arbitrary finite groups. 
\begin{theorem}{\label{line power graph}}
    Let $G$ be a finite group. Then $\po$ is a line graph of some graph $\Gamma$ if and only if $G$ is a cyclic group of prime power order.
\end{theorem}
\begin{theorem}{\label{arbitrary power graph}}
    Let $G$ be a finite non-cyclic group which is not a generalized quaternion group. Then $\ppo$ is a line graph if and only if $G$ satisfies the following conditions: \begin{itemize}
        \item[(i)] For each $M\in \m$, we have $\left|M\right| \in\left\{6, p^\alpha\right\}$ for some prime $p$.
        \item[(ii)] $\mathrm{(a)}$  If $M_i, M_j, M_k \in \mathcal{M}^{(2)}(G)$, then $\left|M_i \cap M_j\right| \leq 2$ and $\left|M_i \cap M_j \cap M_k\right|=1$.\\ 
        $\mathrm{(b)}$ If $M_s \in \m \setminus \mathcal{M}^{(2)}(G)$ and $ M_t \in \m$, then $\left|M_s \cap M_t\right|=1$.
    \end{itemize}
\end{theorem}
\begin{corollary}
    Let $G$ be a finite non-cyclic group of odd order and let $M_i \in \m$. Then $\ppo$ is a line graph if and only if $G$ satisfies the following conditions: 
    \begin{itemize}
        \item[(i)] For each $i$, we have $\left|M_i\right|= p^\alpha$ for some prime $p$.
        \item[(ii)] The intersection of any two maximal cyclic subgroups is trivial.
    \end{itemize}

\end{corollary}
\begin{theorem}{\label{proper power q4n}}
    Let $G$ be a generalized quaternion group $Q_{4n}$. Then $\ppo$ is a line graph if and only if $n\in \{p, 2^k\}$  for some odd prime $p$ and $k\geq 1$.
\end{theorem}
% \begin{theorem}
%     Let $G$ be a finite non-cyclic group which is not a generalized quaternion group, and $M_i\in\m$. Then $\ppo$ is a line graph of some graph $\Gamma$ if and only if $G$satisfies the following conditions:
%     \begin{itemize}
%         \item[(i)] For each $i$, we have $|M_i|\in \{6,p^{\alpha}\}$ for some prime $p$.
%         \item[(ii)] If $M_i, M_j$ and $M_k$ are $2$-groups, then $|M_i\cap M_j|=1$ and $|M_i\cap M_j\cap M_k|=1$. Otherwise, $|M_s\cap M_t|=1$. 
%     \end{itemize}
% \end{theorem} 
Further, we consider the (proper) enhanced power graph of a finite group and classify the finite groups $G$ such that the graphs $\c$ and $\dd$ are line graphs.

\begin{theorem}{\label{line EPG}}
     Let $G$ be a finite group. Then $\c$ is a line graph of some graph $\Gamma$ if and only if $G$ is a cyclic group.
\end{theorem}

\begin{theorem}{\label{arbitrary M EPG}}
    Let $G$ be a finite non-cyclic group and $\t= \displaystyle{\bigcap_{M_{i}\in \m} M_{i}}$. Then $\dd$ is a line graph of some graph $\Gamma$ if and only if $G$ satisfies the following conditions:
    \begin{itemize}
        \item[(i)] $|(M_i\cap M_j) \setminus \t|\leq 1$
        \item[(ii)] $|(M_i\cap M_j\cap M_k) \setminus \t|=0$ 
    \end{itemize}
    % where $M_i, M_j, M_k \in \m $ and $T=\bigcap_{M\in \m} M$.
\end{theorem}
\begin{theorem}{\label{nilpotent proper EPG}}
    Let $G$ be a finite non-cyclic nilpotent group (except  non-abelian $2$-groups). Then $\dd$ is a line graph of some graph $\Gamma$ if and only if $G$ is isomorphic to one of the following groups.
    \begin{itemize}
        \item[(i)] $\mathbb{Z}_2\times \mathbb{Z}_{2^2}$ 
        \item[(ii)]  $\mathbb{Z} _{2^2}\times \mathbb{Z}_{2^2}$ 
        \item[(iii)] $\mathbb{Z}_n\times \mathbb{Z}_p\times \mathbb{Z}_p \times \cdots \times\mathbb{Z}_p$, where $p$ is a prime and $\mathrm{gcd}(n,p)=1$. 
        \item[(iv)] $\mathbb{Z}_n\times Q_{2^k}$ such that $\mathrm{gcd}(2,n)=1$.
         \item[(v)] $\mathbb{Z}_n\times P$ such that $P$ is a non-abelian $p$-group with $\mathrm{gcd}(n,p)=1$ and the intersection of any two maximal cyclic subgroups of $P$ is trivial.
          
    \end{itemize}
\end{theorem}
 Finally, we investigate the question: When  these above mentioned graphs are the complement of line graphs? Consequently, we obtain the following results.
\begin{theorem}{\label{line complement power graph}}
    The power graph $\po$ of a finite group  is the complement of a line graph of some graph $\Gamma$ if and only if $G$ is isomorphic to one of the groups:
    $\mathbb Z_6, \  \mathbb{Z}_2\times \cdots \times \mathbb{Z}_2, \; Q_8, \; \mathbb{Z}_{p^\alpha}, \text{ where $p$ is a prime.}$
\end{theorem}
\begin{theorem}{\label{line complement proper power graph}}
    Let $G$ be a finite group which is not a cyclic $p$-group. Then $\ppo$ is the complement of a line graph of some graph $\Gamma$ if and only if $G$ is isomorphic to one of the groups:
    $\mathbb Z_6, \mathbb{Z}_2\times \cdots \times \mathbb{Z}_2, Q_8.$
\end{theorem}
\begin{theorem}{\label{line complement EPG}}
    Let $G$ be a finite group of order $n$. Then the enhanced power graph $\c$ is the complement of a line graph of some graph $\Gamma$ if and only if $G$ is isomorphic to one of the groups:
    $\mathbb Z_n,   \mathbb{Z}_2\times \cdots \times \mathbb{Z}_2, \; Q_8$.
\end{theorem}
\begin{theorem}{\label{line complement proper EPG}}
     Let $G$ be a finite non-cyclic group. Then $\dd$ is the complement of a line graph of some graph $\Gamma$ if and only if $G$ is isomorphic to either    $\mathbb{Z}_2\times \cdots \times \mathbb{Z}_2$ or $Q_8$.
\end{theorem}
%%%%%%%%%%%%%%%%%%%%%%%%%%%%%%%%%%%%%%%%%%%%%%%%%%%%%%%%%%%%%%%%%%%%%%%%%%%%%%%%%%%%%%%%%%%%%%%%%%%%%%%%%%%%%%%%%%%%%%%%%
% \section{Preliminaries}

%%%%%%%%%%%%%%%%%%%%%%%%%%%%%%%%%%%%%%%%%%%%%%%%%%%%%%%%%%%%%%%%%%%%%%%%%%%%%%%%%%%%%%%%%%%%%%%%%%%%%%%%%%%%%%%%%%%%%%%%%%
\section{Proof of the main results}
In this section, we provide the proof our main results. 
\subsection*{Proof of the Theorem \ref{line power graph}}
In order to prove the Theorem \ref{line power graph}, first we establish the following proposition.

% \subsection*{When power graphs are line graphs}
\begin{proposition}{\label{pg and epg proposition}}
    Let $G$ be a finite non-cyclic group and let $\del \in \{\po , \c \}$. Then there does not exist any graph $\Gamma$ such that $\del = L(\Gamma)$.
\end{proposition}
 \begin{proof}
     Since $G$ is a finite non-cyclic group, by Lemma \ref{M(G) >2}, we have $|\m|\geq 3$. Let $ H_1=\langle x \rangle, \; H_2= \langle y \rangle$ and $H_3= \langle z \rangle $ be three maximal cyclic subgroups of $G$. Then the subgraph induced by the set $\{x,y,z,e\}$  of $\del$ is isomorphic to $K_{1,3}$ (see Remark \ref{remark}(iv)). Consequently, $\del$ is not a line graph of any graph (cf. Lemma \ref{induced lemma}).
 \end{proof}
 On combining Proposition \ref{pg and epg proposition} and {\rm \cite[Theorem 2.1]{a.bera2022}}, we obtain  Theorem \ref{line power graph}.
\subsection*{Proof of Theorem \ref{arbitrary power graph}}
    Let $G$ be a finite non-cyclic group which is not  generalized quaternion. Suppose $\ppo$ is a line graph of some graph $\Gamma$. Then by Theorem \ref{dominating power graph}$, V\left(\ppo \right)=G \setminus \{e\}$. On contrary assume that $G$ does not satisfy condition (i). Then $G$ has a maximal cyclic subgroup $M$ such that neither $|M|=6$ nor $|M|$ is a prime power. Let $|M|=p_1^{\alpha_1} p_2^{\alpha_2} \cdots p_k^{\alpha_k} (k \geq 2)$ be the prime power factorization of $|M|$. If $p_i>3$ for some $i \in[k]$, then $M$ contains at least $4$ elements of order $p_i$ and $4$ elements of order $p_i p_j$ for some $j \in[k] \setminus\{i\}$.
    Let $x_1, x_2, x_3, y_1, z_1 \in M$ such that $o\left(x_1\right)=o\left(x_2\right)=o\left(x_3\right)=p_i p_j$, $o\left(y_1\right)=p_j$ and $o(z)=p_i$. Then by Remark \ref{remark}(ii), the subgraph induced by the set $\left\{x_1, x_2, x_3, y_1, z_1\right\}$ of $\ppo$ is isomorphic to $\Gamma_3$; a contradiction (see Lemma \ref{induced lemma}). Thus, $p_i \leq 3$ for all $i \in[k]$. Therefore, $|M|=2^\alpha 3^\beta$ for some $\alpha \geq 1, \beta \geq 1$. Since $|M| \neq 6$, we get either $\alpha \geq 2$ or $\beta \geq 2$. First suppose that $\alpha \geq 2$. Let $x_1, x_2, y_1, y_2, z_1, z_2 \in M$ such that $o\left(x_1\right)=o\left(x_2\right)=4,\;  o\left(y_1\right)$ $=o\left(y_2\right)=3$ and $o\left(z_1\right)=o\left(z_2\right)=12$. Observe that the subgraph induced by the set $\left\{x_1, x_2, y_1, y_2, z_1, z_2\right\}$ of $\ppo$ is isomorphic to $\Gamma_6$; again a contradiction. Similarly,  we get a contradiction for $\beta \geq 2$. Thus, for each $M\in \m$, we obtain $|M|\in \{6, p^{\alpha}\}$. \\
    If possible, assume that $G$ does not satisfy condition (ii). Further, we have the following cases: \\
 \noindent\textbf{Case-1:} \emph{$G$ does not satisfy} (ii)(a). Then we have the following further two subcases:
 
\textbf{Subcase-1.1:} \emph{$\left|M_1 \cap M_2 \cap M_3\right| \geq 2$ for some $M_1, M_2, M_3 \in \mathcal{M}^{(2)}(G)$.}
Let $x \; (\neq e) \in M_1 \cap M_2 \cap M_3$ and $M_i=\left\langle m_i\right\rangle$ for each $i \in\{1,2,3\}$. Then by Remark \ref{remark}(ii), the subgraph induced by the vertex set $\left\{x, m_1, m_2, m_3\right\}$ of $\ppo$ is isomorphic to $K_{1,3}$, which is not possible.

\textbf{Subcase-1.2:} \emph{$| M_1 \cap M_2 | \geq 3$ for some $M_1, M_2 \in \mathcal{M}^{(2)}(G)$.} Suppose $x, y \in (M_1\cap  M_2)\setminus \{e\}$. Let $M_1=\left\langle m_1\right\rangle$ and $M_2=\left\langle m_2\right\rangle$. By Remark \ref{remark}, the subgraph induced by the set $\left\{x, y, m_1,m_1^{-1}, m_2, m_2^{-1}\right\}$ of $\ppo$ is isomorphic to $\Gamma_6 ;$ a contradiction. \\
Thus, $G$ must satisfy the condition (ii)(a).\\
\noindent\textbf{Case-2:} \emph{$G$ does not satisfy} (ii)(b). Then there exist two maximal cyclic subgroups $M_1 \in \m \setminus \mathcal{M}^{(2)}(G)$ and $M_2\in \m$ such that $\left|M_1 \cap M_2\right| \geq 2$. In view of the condition (i), we discuss the following subcases. 

\textbf{Subcase-2.1:} $\left|M_1\right|=6$. Then $\left|M_1 \cap M_2\right| \in\{2,3\}$. Let $\left|M_1 \cap M_2\right|=2$, $M_1=\langle x\rangle$ and $M_2=\langle y\rangle$. Then $x^3 \in M_1 \cap M_2$ because $x^3$ is the only element of order $2$ in $M_1$. The subgraph induced by the set $\left\{ x, x^2, x^3, x^4, x^5, y\right\}$ of $\ppo$ is isomorphic to $\Gamma_5$; a contradiction. If $\left|M_1 \cap M_2\right|=3$, then $x^2, x^4 \in M_1 \cap M_2$. The subgraph induced by the set $\left\{x, x^2, x^4, x^5, y, y^{-1}\right\}$ is isomorphic to $\Gamma_6$, which is not possible.

\textbf{Subcase-2.2:} $|M_1|=p^\alpha  (p> 2)$. Let $M_1=\langle x\rangle$, $M_2=\langle y\rangle$ and let $m\; (\neq e)\in M_1 \cap M_2 $. Then  the subgraph of $\ppo $ induced by the set $\{ x, x^{-1}, y, y^{-1}, m, m^{-1}\}$  is isomorphic to $\Gamma_6$; a contradiction.

Conversely, suppose that $G$ satisfies both the given conditions. On contrary, assume that $\ppo$ is not a line graph and so it has an induced subgraph $\Gamma$ isomorphic to one of the  nine graphs given in Figure \ref{fig line graph}. In view of Remark \ref{remark maximal}, first we prove the following claim.
\vspace{.1cm} \\
\textit{Claim $2.3$:} If $x\in V(\Gamma)$ such that $x\in M$ for some $M\in \m$, then $M$ must be a $2$-group.\\
\textit{Proof of claim:} If possible, let $x\in M$ and $M\in \m \setminus \mathcal{M}^{(2)}(G)$. Then by condition (ii)(b), we have $M\cap M'=\{e\}$ for all $M'\; (\neq M)\in \m$. Consequently, by Remark \ref{remark}(ii), we get $N(x) \subseteq M$ in $\ppo$. Therefore, if $x\sim y$, then $y\in M$ and so $N(y) \subseteq M$. Connectedness of $\Gamma$ implies that $V(\Gamma) \subseteq M$. For $|M|=p^\alpha$, where
$p$ is an odd prime, note that the subgraph $\ppo$ induced by any non-empty subset of $M$ is a complete graph. It implies that $\Gamma$ is a complete subgraph of $\ppo$ which is not true because $\Gamma \cong \Gamma_i$ for some $i$, where $1\leq i\leq 9$ (see Figure \ref{fig line graph}).\\
If $|M|=6$, then $M \cong \mathbb{Z}_6$. Observe that the subgraph of $\ppo$ induced by the set $M \backslash\{e\}$, shown in Figure \ref{fig Z6 power}, can not contain $\Gamma$ as an induced subgraph. Thus, the claim holds.
 \begin{figure}[ht]
    \centering
    \includegraphics[scale=.9]{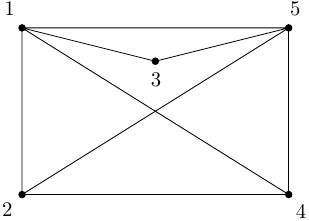}
    \caption{Subgraph induced by $M\setminus \{e\}$ of $\ppo$.}
    \label{fig Z6 power}
\end{figure}
%In view of the claim $2.3$, we obtain that for each $x \in V(\Gamma)$, there exists a maximal cyclic subgroup $M \in \mathcal{M}^{(2)}(G)$ such that $x \in M$.

Now if $\Gamma$ is isomorphic to $K_{1,3}$, as shown in Figure \ref{gamma},  then by Remark \ref{remark}(ii) and Claim $2.3$, there exist maximal cyclic subgroups $M_1, M_2, M_3 \in \mathcal{M}^{(2)}(G)$ such that $a, d \in M_1,\;  b, d \in M_2$ and $c, d \in M_3$. Note that $M_1 \neq M_2$. Otherwise, $a\sim b $ in $\ppo$ (see Remark \ref{remark}(ii)), which is not possible. Similarly, $M_2 \neq M_3$ and $M_1 \neq M_3$. Also, $d \in (M_1 \cap M_2 \cap M_3)\setminus \{e\}$; a contradiction to the condition (ii)(a). Thus, $\Gamma$ can not be isomorphic to $K_{1,3}$.

Now, suppose $\Gamma\cong \Gamma_i$ for some $i$, $2\leq i \leq 9$. Further, note that  $\Gamma$ has an induced subgraph isomorphic to $\Gamma'$ as shown in Figure \ref{gamma}.

\begin{figure}[ht]
    \centering
    \includegraphics[scale=.9]{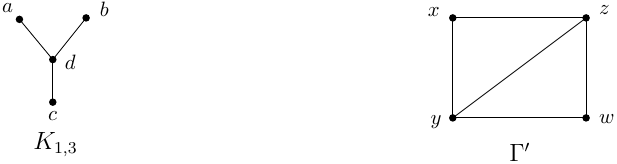}
    \caption{}
    \label{gamma}
\end{figure}
Since $x \sim y, \; y \sim z$ and $z \sim x$,  then by the definition of $\ppo$, it is easy to observe that  one of the following three holds: $x, y \in\langle z\rangle,\;  y, z \in\langle x\rangle,\;  x, z \in\langle y\rangle$. Consequently, there exists $M \in \mathcal{M}^{(2)}(G)$ such that $x, y, z \in M$. Similarly, there exists $M^{\prime} \in \mathcal{M}^{(2)}(G)$ such that $y, z, w \in M^{\prime}$. Notice that $M \neq M^{\prime}$.
Otherwise, $x \sim w$ in $\ppo$, which is not possible. But $y, z \in M \cap  M^{\prime}$;  a contradiction of condition (ii)(a).
Thus, $\ppo$ is a line graph. This completes our proof.

\vspace{.3cm}
For the dihedral group $D_{2n}= \langle a, b  :  a^{n} = b^2 = e,  ab = ba^{-1} \rangle$, note that $\mathcal{P}^{**}(D_6) = L(K_{1,2}\cup 3 K_2)$. It follows that Theorem $1.10$ of \cite{a.bera2022} is not correct. Moreover, we correct the same in the following corollary.
\begin{corollary}
    Let $G$ be the dihedral group $D_{2 n}$ of order $2n$. Then $\ppo$ is a line graph of some graph $\Gamma$ if and only if $n\in \{6,  p^\alpha\}$ for some prime $p$.
 \end{corollary}
 \begin{proof}
          First assume that $\mathcal{P}^{**}\left(D_{2 n}\right)$ is a line graph. Note that $D_{2 n}$ has one maximal cyclic subgroup $M=\langle a \rangle$ of order $n$, and $n$ maximal cyclic subgroups $M_i=\langle a^ib\rangle$, where $1\leq i \leq n$, of order $2$. Then by Theorem \ref{arbitrary power graph}, either $n=6$ or $p^\alpha$.
Conversely, if $n\in \{6,  p^\alpha\}$, then $G$ satisfies the condition (i). Note that the intersection of any two maximal cyclic subgroups of $D_{2n}$ is trivial. Thus, condition (ii) holds. By Theorem \ref{arbitrary power graph}, $\mathcal{P}^{**}\left(D_{2 n}\right)$ is a line graph.
\end{proof}
  \begin{corollary}
      Let $G$ be the  semidihedral group $SD_{8n} = \langle a, b  :  a^{4n} = b^2 = e,  ba = a^{2n -1}b \rangle$. Then $\mathcal{P}^{**}(SD_{8n})$ is not a line graph of any graph.
  \end{corollary}
 \begin{proof}
     Since $M_1=\langle a\rangle$, $M_2=\langle ab\rangle$ and $M_3=\langle a^3b\rangle$ are three maximal cyclic subgroups of $SD_{8n}$ such that  $M_1\cap M_2\cap M_3=\{e, a^{2n}\}$. By Theorem \ref{arbitrary power graph}, $\mathcal{P}^{**}(SD_{8n})$ is not a line graph of any graph.

 \end{proof}

% \begin{theorem}
%     Let $G$ be a generalized quaternion group $Q_{4n}$. Then $\ppo$ is a line graph if and only if $n=2^k$ or $p$ for some prime $p$.
% \end{theorem}
\subsection*{Proof of Theorem \ref{proper power q4n}}
    Let $G$ be a generalized quaternion group such that $\ppo$ is a line graph. Then $V(\ppo)=G\setminus Z(G)$ and  $G$ has one maximal cyclic subgroup of order of $2 n$ and $n$ maximal cyclic subgroups of order $4$. Let $M$ be the maximal cyclic subgroup of order $2 n$. If possible, assume that $n$ is divisible by two primes $p$ and $q$ such that $p< q$. Let $x_1, x_2, y_1, y_2, z_1, z_2 \in M$ such that $o\left(x_1\right)=o\left(x_2\right)=2 p, \;o\left(y_1\right)=o\left(y_2\right)=2 p q$ and $o\left(z_1\right)=o\left(z_2\right)=q$. By Remark \ref{remark}(ii), the subgraph induced by the set $\left\{x_1, x_2, y_1, y_2, z_1, z_2\right\}$ is isomorphic to $\Gamma_6$; a contradiction. Thus, $n=p^\alpha$.  If $p>2$ and $\alpha \geq 2$, then note that $M$ has at least two elements $x, x'$ of order $2 p$, two elements $y,y'$  of order $p^2$ and two elements $z, z'$ of order $2 p^2$. The induced subgraph by the set $\{x,x',y,y',z,z'\}$ is isomorphic to $\Gamma_6$, again a contradiction. Thus, either $n=2^k$ or $n$ is an odd prime.
 Conversely, If $n$ is an odd prime, then
$
\mathcal{P}^{**}\left(Q_{4 n}\right)=K_{2 n-2} \cup n K_2=L\left(K_{1,2 n-2} \cup n K_{1,2}\right) \text {. }$ If $n=2^k$, then $\mathcal{P}^{**}\left(Q_{4 n}\right)$ is a line graph (cf. {\rm \cite[Theorem 1.9]{a.bera2022}}). This completes the proof.

\vspace{.3cm}
Now we intend to classify all the groups $G$ such that the enhanced power graph $\c$ and the proper enhanced power graph $\dd$ are line graphs, respectively. In order to prove the Theorem \ref{line EPG}, first note that if $G$ is a cyclic group of order $n$, then by Theorem \ref{complete}, we get $\c \cong K_n$. Further,  note that 
$K_n=L(K_{1 , n})$. Consequently, we have the following lemma.
 \begin{lemma}{\label{epg complete lemma}}
    Let $G$ be a finite cyclic group. Then $\c$ is a line graph of some graph. 
 \end{lemma}
 Thus, by Proposition \ref{pg and epg proposition} and Lemma \ref{epg complete lemma}, we obtain Theorem \ref{line EPG}. \\
% \begin{theorem}{\label{arbitrary power graph}}
%     Let $G$ be a finite non-cyclic group which is not a generalized quaternion group, and let $M_i \in \m$. Then $\ppo$ is a line graph if and only if $G$ satisfies the following conditions: \begin{itemize}
%         \item[(i)] For each $i$, we have $\left|M_i\right| \in\left\{6, p^\alpha\right\}$ for some prime $p$.
%         \item[(ii)] (a)  If $M_i, M_j, M_k \in \mathcal{M}^{(2)}(G)$, then $\left|M_i \cap M_j\right| \leq 2$ and $\left|M_i \cap M_j \cap M_k\right|=1$.\\ 
%         (b) If $M_s \in \m \setminus \mathcal{M}^{(2)}(G), M_t \in M(G)$, then $\left|M_s \cap M_t\right|=1$.
%     \end{itemize}
% \end{theorem}

If $G$ is cyclic, then $\dd$ is an empty graph (cf. Theorem \ref{complete}). Consequently, we now characterize all the  finite non-cyclic groups $G$ whose proper enhanced power graphs $\dd$ are line graphs (see Theorem \ref{arbitrary M EPG}).

\subsection*{Proof of  Theorem \ref{arbitrary M EPG}}

    First, suppose that $\dd$ is a line graph of some graph $\Gamma$. On contrary, assume that $G$ does not satisfy condition (i). Then $G$ has two maximal cyclic subgroups $M_1$ and $M_2$ such that $|(M_1\cap M_2) \setminus \t |\geq 2$. Since $e\in \t $, we have $|M_1|\geq 3$ and $|M_2| \geq 3$. Suppose $M_1 =\langle x_1 \rangle = \langle y_1 \rangle$ and  $M_2= \langle x_2 \rangle = \langle y_2 \rangle$. Further, let $x,y \in (M_1\cap M_2) \setminus \t $. Then the subgraph induced by the set $\{x,y,x_1,y_1,x_2, y_2\}$ is isomorphic to $\Gamma_6$ (see Remark \ref{remark}); a contradiction. Thus, $G$ must satisfy the condition (i). Now suppose that $G $ does not satisfy the condition (ii). Then $G$ has three maximal cyclic subgroups $M', M'', M'''$ such that $|(M'\cap M''\cap M''') \setminus \t |\geq 1$. Assume that $m\in (M'\cap M''\cap M''') \setminus \t  $. Consider $ M'= \langle m' \rangle, \;  M''= \langle m'' \rangle$ and $ M'''= \langle m''' \rangle $.  Then the subgraph induced by the set $\{ m, m', m'', m'''\}$ is isomorphic to $\Gamma_1$ (cf. Remark \ref{remark}); a contradiction. 

     Conversely, suppose that $G$ satisfies (i) and (ii). On contrary assume that $\dd$ is not a line graph. Then by Lemma \ref{induced lemma}, $\dd$ has an induced subgraph isomorphic to one of the nine graphs given in Figure \ref{fig line graph}. Let $\dd$ has an induced subgraph isomorphic to $K_{1,3}$ given in Figure \ref{gamma}.  Consequently, $\langle a,d \rangle , \langle b,d \rangle$ and $\langle c , d \rangle$ are cyclic subgroups of $G$. Let $M_1, M_2$ and $M_3$ be maximal cyclic subgroups containing  $\langle a,d \rangle , \langle b,d \rangle$ and $\langle c , d \rangle$, respectively. Note that $M_1\neq M_2$. Otherwise, $a\sim b$ in $\dd$. Similarly, $M_2\neq M_3$ and $M_3\neq M_1$. Since $d\in V(\dd)$, we obtain $d\notin \t $. It follows that $d\in (M_1\cap M_2\cap M_3) \setminus \t $; a contradiction of condition (ii). Thus, $\dd$ can not contain an induced subgraph isomorphic to $K_{1,3}$. Now suppose that $\dd$ has an induced subgraph isomorphic to one of the remaining eight graphs in Figure \ref{fig line graph}. Then observe that $\dd$ has an induced subgraph isomorphic to $\Gamma '$ as shown in Figure \ref{gamma}. Note that that $x, y$ and $z$ belong to a maximal cyclic subgroup of $G$. On contrary, assume that $x, y, z \notin M$ for any $M\in \m$. Since $x\sim y, y\sim z$ and $z\sim x$ in $\dd$, by Remark \ref{remark}(i), we have three maximal cyclic subgroups $M_4, M_5$ and $M_6$ such that $x,y \in M_4, \  y,z \in M_5$ and $z,x\in M_6$. Thus, $x\in (M_4\cap M_6)\setminus \t $. If $o(x)\geq 3$, then $ x^{-1}\; (\neq x)\in M_4\cap M_6$. Further, note that $N(x)=N(x^{-1})$ and so $x^{-1}\notin \t $. It follows that $x^{-1}\in (M_4\cap M_6)\setminus \t $; a contradiction of condition (i). Consequently, $o(x)=2$. Similarly, $o(y)=o(z)=2$. But $M_4$ cannot contain two elements of order $2$. Thus, $x, y, z \in M' $ for some $M'\in \m$. By  similar argument, we get $y, z, w\in M''$ for some $M''\in \m$. Note that $M'\neq M''$. Otherwise, $x\sim w$ in $\dd$. Also, $y,z\in (M'\cap M'')\setminus \t $; again a contradiction. Thus, $\dd$ cannot contain an induced subgraph isomorphic to the graph $\Gamma '$ (see Figure \ref{gamma}). This completes our proof.

\begin{example}
     For $n \geq 2$, consider the \emph{semidihedral group} $SD_{8n} = \langle a, b  :  a^{4n} = b^2 = e,  ba = a^{2n -1}b \rangle.$
     %Then $SD_{8n}$ has one maximal cyclic subgroup of order $4n$, $n$ maximal cyclic subgroups of order $4$ and $2n$ maximal cyclic subgroups of order $2$. 
     Since $SD_{8n}$ has a maximal cyclic subgroup $M=\langle a^2b \rangle$ of order $2$, therefore $\mathcal{T}(SD_{8n})=\{e\}$.  Consider $M_1=\langle a\rangle$, $M_2=\langle ab\rangle$ and $M_3=\langle a^3b\rangle$ are three maximal cyclic subgroups of $SD_{8n}$. Then note that $M_1\cap M_2\cap M_3=\{e, a^{2n}\}$. Thus, $SD_{8n}$ does not satisfy the condition (ii) of Theorem \ref{arbitrary M EPG}, and so $\mathcal{P}^{**}_E(SD_{8n})$ is not a line graph of any graph.
\end{example}
\begin{corollary}{\label{any order corollary}}
    Let $G$ be a finite non-cyclic group such that the intersection of any two maximal cyclic subgroups is equal to $\t $. Then $\dd = L(\Gamma)$ for some graph $\Gamma$.
\end{corollary}
\begin{example}
    For $n \geq 2$, the \emph{generalized quaternion group} $Q_{4n} = \langle a, b  :  a^{2n}  = e, a^n= b^2, ab = ba^{-1} \rangle.$
Then $Q_{4n}$ has one maximal cyclic subgroup of order $2n$ and $n$ maximal cyclic subgroups of order $4$. Observe that the intersection of any two maximal cyclic subgroups of $Q_{4n}$ is $\{e, a^n\}$ and so $\mathcal{T}(Q_{4n})=\{e, a^n\}$. Consequently, $P^{**}_E(Q_{4n})$ is a line graph of some graph $\Gamma$. Indeed, $P^{**}_E(Q_{4n})=K_{2n-2}\cup nK_2= L(K_{1,2n-2}\cup nK_{1,2})$.
\end{example}
\begin{example}
    For $n \geq 3$, consider the \emph{dihedral group} $D_{2n} = \langle a, b  :  a^{n} = b^2 = e,  ab = ba^{-1} \rangle.$
Then $D_{2n}$ has one maximal cyclic subgroup of order $n$ and $n$ maximal cyclic subgroups of order $2$. Consequently, the intersection of any two maximal cyclic subgroups of $D_{2n}$ is trivial. It follows that $\mathcal{T}(D_{2n})=\{e\}$. Thus, by Corollary \ref{any order corollary}, $P^{**}_E(D_{2n})$ is a line graph of some graph $\Gamma$. In fact, $P^{**}_E(D_{2n})= K_{n-1}\cup nK_1= L(K_{1,n-1}\cup nK_2)$.
\end{example}
The converse of the  Corollary \ref{any order corollary} need not be true in general. For instance, if $G= \mathbb{Z}_2 \times \mathbb{Z}_{2^2}$, then $\dd$ is a line graph of some graph but $\mathbb{Z}_2 \times \mathbb{Z}_{2^2}$ has two maximal cyclic subgroups whose intersection is non-trivial. However, if $G$ is of odd order, then the converse is also true.

\begin{corollary}{\label{odd order condition}}
    Let $G$ be a finite group of odd order. Then $\dd$ is a line graph of some graph $\Gamma$ if and only if the intersection of any two maximal cyclic subgroups is equal to $\t $.
\end{corollary}
\begin{proof}

    Suppose that $\dd$ is a line graph of some graph $\Gamma$. On contrary, assume that there exist two maximal cyclic subgroups $M_1$ and $M_2$ such that $x\in (M_1\cap M_2)\setminus \t $. Since $M_1\cap M_2$ is a subgroup of $G$, we obtain  $x^{-1}\in M_1\cap M_2$. Also, $N(x)=N(x^{-1})$. It follows that $x^{-1}\in (M_1\cap M_2)\setminus \t $; a contradiction (see Theorem \ref{arbitrary M EPG}). 
\end{proof}
% \begin{corollary}
%     Let $G$ be a finite non-cyclic group such that the intersection of any two maximal cyclic subgroups of $G$ is trivial. Then $\dd$ is a line graph of some graph $\Gamma$.
% \end{corollary}

\subsection*{Proof of Theorem \ref{nilpotent proper EPG}} In order to prove Theorem \ref{nilpotent proper EPG}, first we prove some necessary results.
\begin{lemma}{\label{one non cyclic}}
    Let $G$ be a finite non-cyclic nilpotent group. If  $\dd$ is a line graph, then there exists a unique Sylow subgroup of $G$ which is non-cyclic.
\end{lemma}
\begin{proof}
Let $G=P_1P_2  \cdots  P_r$ be a finite nilpotent group such that $P_i$'s are Sylow $p_i$-subgroups of $G$. On contrary, assume that $G$ has two Sylow subgroups which are non-cyclic. Without loss of generality, suppose that $P_1$ and $P_2$ are non-cyclic. It implies that $\mathcal{M}(P_i)\geq 3$ for every $i\in \{1,2\}$ (cf. Lemma \ref{M(G) >2}). Consider $M_1, M_1', M_1'' \in \mathcal{M}(P_1)$ such that 
$M_1= \langle x_1\rangle $, $M_1'=\langle y_1\rangle$ and $M_1''=\langle z_1 \rangle $. Since $G$ is non-cyclic, we get $M, M', M''\in \m$ where $M=M_1M_2\cdots M_r,\; M'=M_1'M_2\cdots M_r$ and $M''=M_1''M_2\cdots M_r$ (cf. Lemma \ref{on maximal connectivity}). Assume that $M_i =\langle x_i\rangle$ for $i\in [r]\setminus \{1\}$. Note that $M= \langle x\rangle  ,\; M'= \langle y\rangle $ and  $M''=\langle z \rangle $, where $x=x_1x_2\cdots x_r, \; y=y_1x_2\cdots x_r$ and   $z=z_1x_2\cdots x_r$. By Remark $\ref{remark}$, $x\nsim y$, $y\nsim z$ and $z\nsim x$ and so $x,y,z\in V(\dd)$. Now consider 
 $t=e_1x_2\cdots x_r$ and $t'=e_1x_2'x_3\cdots x_r$, where $e_1$ is the identity element of $P_1$,  $\langle x_2\rangle \in \mathcal{M}(P_2)$ and $\langle x_2'\rangle \neq \langle x_2\rangle $. By Remark \ref{remark}(iv), we have $x_2\nsim x_2'$ in $\mathcal{P}_E(P_2)$ and so  $t'\nsim t $ in $\c$ (cf. {\rm \cite[Theorem 5.4]{a.kumar2023}}). Also, $x\sim t$, $y\sim t$ and $z\sim t$  in $\c$ and so in $\dd$.
Thus, the subgraph induced by the set $\{x,y,z,t\}$ is isomorphic to $\Gamma_1$ (see Figure \ref{fig line graph}); a contradiction. Thus, the result holds.
\end{proof}
\begin{proposition}{\label{proper EPG abelian}}
     Let $G$ be a finite non-cyclic abelian group. Then $\dd$ is a line graph of some graph $\Gamma$ if and only if $G$ is isomorphic to one of the following groups:
    \begin{itemize}
        \item[(i)] $\mathbb{Z}_2\times \mathbb{Z}_{2^2}$ 
        \item[(ii)]  $\mathbb{Z} _{2^2}\times \mathbb{Z}_{2^2}$ 
        \item[(iii)] $\mathbb{Z}_n\times \mathbb{Z}_p\times \mathbb{Z}_p \times \cdots \times\mathbb{Z}_p$, where $p$ is a prime and $\mathrm{gcd}(n,p)=1$. 
        \end{itemize}
\end{proposition}
\begin{proof}
    Let $G$  be a finite non-cyclic abelian group. Then $G=P_1\times P_2 \times \cdots \times P_r$, where $P_i$'s are Sylow $p_i$-subgroups of $G$. Suppose that $\dd$ is a line graph. Then by Lemma \ref{one non cyclic}, $G$ has a unique non-cyclic Sylow subgroup. Consequently, $G\cong \mathbb{Z}_n\times P$, where $P$ is a non-cyclic abelian Sylow $p$-subgroup of $G$ and $\mathrm{gcd}(p,n)=1$. Then by Theorem \ref{Dominating enhanced power graph}, $V(\dd )= G\setminus \{(a,e): a\in \mathbb{Z}_n\}$, where $e$ is the identity element of $P$.  Observe that $\mathcal{P}_E(P)$ is an induced subgraph of $\c$. Also, $\mathcal{P}_E(P)= \mathcal{P}(P)$ (cf. {\rm \cite[Theorem 31]{a.Cameron2016}}). By  {\rm \cite[Proposition 3.5]{a.bera2022}}, if $P\ncong \mathbb{Z}_2\times \mathbb{Z}_{2^2}, \;  \mathbb{Z} _{2^2}\times \mathbb{Z}_{2^2},\;  \mathbb{Z}_p\times \mathbb{Z}_p \times \cdots \times\mathbb{Z}_p$, then $\mathcal{P}^{**}(P)$ has an induced subgraph isomorphic to $\Gamma_1$ and so $\dd$ contains an induced subgraph isomorphic to $\Gamma_1$; a contradiction. Thus, $P$ is isomorphic to one of the three groups: $\mathbb{Z}_2\times \mathbb{Z}_{2^2},\;  \mathbb{Z} _{2^2}\times \mathbb{Z}_{2^2}, \; \mathbb{Z}_p\times \mathbb{Z}_p \times \cdots \times\mathbb{Z}_p$. If $n>1$ and $G\cong \mathbb{Z}_n\times \mathbb{Z}_2\times \mathbb{Z}_{2^2}$, then $G$ has two maximal cyclic subgroups $M_1=\langle (\overline{1},\overline{0},\overline{1})\rangle$ and $M_2=\langle (\overline{1},\overline{1},\overline{1})\rangle$ of order $4n$ such that $(\overline{1},\overline{0},\overline{2}), (\overline{2},\overline{0},\overline{2})\in (M_1\cap M_2)\setminus \t $, where $\t =\{(a,\overline{0},\overline{0}): a\in \mathbb{Z}_n\}$; a contradiction (see Theorem \ref{arbitrary M EPG}). Thus, $n=1$.

    If $n>1$ and $G\cong \mathbb{Z}_n\times \mathbb{Z}_{2^2}\times \mathbb{Z}_{2^2}$, then $G$ has two maximal cyclic subgroups $M_3=\langle (\overline{1},\overline{1},\overline{0})\rangle$ and $M_4=\langle (\overline{1},\overline{1},\overline{2})\rangle$ of order $4n$ such that $(\overline{1},\overline{2},\overline{0}), (\overline{2},\overline{2},\overline{0})\in (M_3\cap M_4)\setminus \t $, where $\t =\{(a,\overline{0},\overline{0}): a\in \mathbb{Z}_n\}$; again a contradiction.

    Conversely, if either $G\cong \mathbb{Z}_2\times \mathbb{Z}_{2^2}$ or $G \cong \mathbb{Z}_{2^2}\times \mathbb{Z}_{2^2}$, then $\dd= \ppo$. By {\rm \cite[Theorem 3.4]{a.bera2022}}, $\dd = L(\Gamma)$ for some graph $\Gamma$.  Now suppose $G\cong \mathbb{Z}_n\times \mathbb{Z}_p\times \mathbb{Z}_p \times \cdots \times\mathbb{Z}_p \; (k\text{-times})$, where $k\geq 2$. Then $\dd = \frac{p^k-1}{p-1}K_{(p-1)n} = L(\frac{p^k-1}{p-1}K_{1,(p-1)n})$. This completes the proof.
\end{proof} 
 \begin{proposition}{\label{Proper EPG non-abelian nilpotent}}
     Let $G$ be a finite non-abelian nilpotent group (except non-abelian $2$-group). Then $\dd$ is a line graph of some graph $\Gamma$ if and only if $G$ is isomorphic to one of the following groups:
     \begin{itemize}
        \item[(i)] $\mathbb{Z}_n\times Q_{2^k}$ such that $\mathrm{gcd}(2,n)=1$.
         \item[(ii)] $\mathbb{Z}_n\times P$ such that $P$ is a non-abelian $p$-group with $\mathrm{gcd}(n,p)=1$ and the intersection of any two maximal cyclic subgroups of $P$ is trivial.
     \end{itemize}
 \end{proposition}
 \begin{proof}
       Let $G=P_1\times P_2\times \cdots \times P_r$ be a finite non-abelian nilpotent group which is not a $2$-group. Suppose that $\dd$ is a line graph. By Lemma \ref{one non cyclic}, exactly one $P_i$ is non-cyclic. Consequently, $G\cong \mathbb{Z}_n\times P$ such that $P$ is a non-abelian $p$-group and $\mathrm{gcd}(n,p)=1$. If $P=Q_{2^k}$, then there is nothing to prove. We may now suppose that $P$ is not a generalized group and $n>1$. On contrary, assume that $P$ has two maximal cyclic subgroups $M'$ and $M''$ such that $x \; (\neq e)\in M'\cap M''$. Consequently, $G$ has two maximal cyclic subgroup $M_1=\mathbb{Z}_n \times M'$ and $M_2=\mathbb{Z}_n \times M''$ (see Lemma \ref{on maximal connectivity}) such that $(\overline{1},x), (\overline{2},x)\in M'\cap M''$. Since $\t = \{(a,e): a\in \mathbb{Z}_n\}$ (see Theorem \ref{Dominating enhanced power graph} and Remark \ref{remark}(i)), we get a contradiction of Theorem \ref{arbitrary M EPG}. Thus, $G$ is isomorphic to the group described in (ii). Now suppose $n=1$ then $G=\mathbb{Z}_1\times P$ is a $p$-group, where $p$ is an odd prime. Then by Corollary \ref{odd order condition}, the intersection of any two maximal cyclic subgroups of $G$ is equal to $\t$. By Theorem \ref{Dominating enhanced power graph}, $\t =\{e\}$. Thus, $G$ is isomorphic to the group described in (ii).

       Conversely, suppose that  $G\cong \mathbb{Z}_n\times Q_{2^k}$ such that $\mathrm{gcd}(2,n)=1$. Also, the intersection of any two maximal cyclic subgroups of $Q_{2^k}$ is $Z(Q_{2^k})$. Consequently, the intersection of any two maximal cyclic subgroups of $G$ is the set $\{(a,b): a\in \mathbb{Z}_n, b\in Z(Q_{2^k})\}$ (see Lemma \ref{on maximal connectivity}). Indeed, $\t = \{(a,b): a\in \mathbb{Z}_n, b\in Z(Q_{2^k})\}$. By Corollary \ref{any order corollary}, $\dd$ is a line graph of some graph $\Gamma$. If $G\cong \mathbb{Z}_n\times P$, where $P$ is a non-abelian $p$-group such that $\mathrm{gcd}(n,p)=1$ and the intersection of any two maximal cyclic subgroups of $P$ is trivial. Then by Lemma \ref{on maximal connectivity}, the intersection of any two maximal cyclic subgroups of $G$ is the set $\{(a,e): a\in \mathbb{Z}_n\}$. Moreover, $\t =\{(a,e): a\in \mathbb{Z}_n\}$. By Corollary \ref{any order corollary}, $\dd$ is a line graph of some graph $\Gamma$.
 \end{proof}
 On combining Proposition \ref{proper EPG abelian} and Proposition \ref{Proper EPG non-abelian nilpotent}, we obtain Theorem \ref{nilpotent proper EPG}.
 % \subsection*{When power graphs are the complement of line graphs}
 \subsection*{Proof of Theorems \ref{line complement power graph}-\ref{line complement proper EPG}}
The following propositions play an important role to prove the Theorems \ref{line complement power graph}, \ref{line complement proper power graph}, \ref{line complement EPG}, \ref{line complement proper EPG}.
% First we prove the following necessary propositions.
 %In this subsection, we prove the  Theorems \ref{line complement power graph}, \ref{line complement proper power graph}, \ref{line complement EPG} with the combination of  Proposition \ref{complement all non-cyclic} and Proposition \ref{cyclic power complement},  \ref{cyclic proper complement}, \ref{complement cyclic EPG  }, respectively. Theorem \ref{line complement proper EPG} is directly followed from Proposition \ref{complement all non-cyclic}.
 \begin{proposition}{\label{cyclic power complement}}
     Let $G$ be a finite cyclic group. Then $\po$ is the complement of a line graph of some graph $\Gamma$ if and only if either $G\cong \mathbb{Z}_6$ or $G\cong \mathbb{Z}_{p^{\alpha}}$ for some prime $p$.
 \end{proposition}
 \begin{proof}
     Let $G$ be a  cyclic group of order $n$. First, suppose that $\po$ is the complement of line graph of some graph $\Gamma$. On contrary assume that neither $n=6$ nor $n$ is a prime power. Consider $n=p_1^{\alpha _1}p_2^{\alpha _2}\cdots p_r^{\alpha _r} \; (r\geq 2)$ is the prime factorization of $n$ such that $p_1<p_2<\cdots <p_r$. Now, if $p_i \geq 5$ for any $i\in [r]$, then $G$ has at least $4$  elements of order $p_i$. Let $x,y ,z,w \in G$ such that $o(x)=o(y)=o(z)=p_i$ and $o(w)=p_j$ for some $j\in [r]\setminus \{i\}$. Then by Remark \ref{remark}(ii), the subgraph induced by the set $\{x,y,z,w\}$ is isomorphic to $\overline{\Gamma_1}$ (see Lemma \ref{induced lemma}); a contradiction. Thus, $p_i\leq 3$ for all $i\in [r]$. Consequently, $r=2$ and $p_1=2, p_2=3$. Since $n\neq 6$, we have either $\alpha _1\geq 2$ or $\alpha _2 \geq 2$. If $\alpha _1\geq 2$, then consider $x_1, x_2 , x_3, x_4 \in G$ such that $o(x_1)=2^{\beta _1}, o(x_2)=2^{\beta_2 }, o(x_3)=2^{\beta_3} $ and $o(x_4)=3$. The subgraph of $\po$ induced by the set $\{x_1, x_2, x_3, x_4\}$ is isomorphic to $\overline{\Gamma_1}$; a contradiction. Similarly, if $\alpha_2 \geq 2$ then again we get a contradiction. Thus, either $G\cong \mathbb{Z}_6$ or $G\cong \mathbb{Z}_{p^{\alpha}}$ for some prime $p$.

  Conversely, if $G\cong \mathbb{Z}_{p^{\alpha}}$, then $\po = K_{p^{\alpha}}$ (cf. Theorem \ref{complete power graph}). Observe that $K_n=\overline{L(n K_2)}$ and so $\po = \overline{L(p^{\alpha} K_2)}$. If $G\cong \mathbb{Z}_6$ then by Figure \ref{fig z6}, we have $\po = \overline{L(3K_2\cup P_4)}$. This completes our proof.
   \begin{figure}[ht]
    \centering
    \includegraphics[scale=.6]{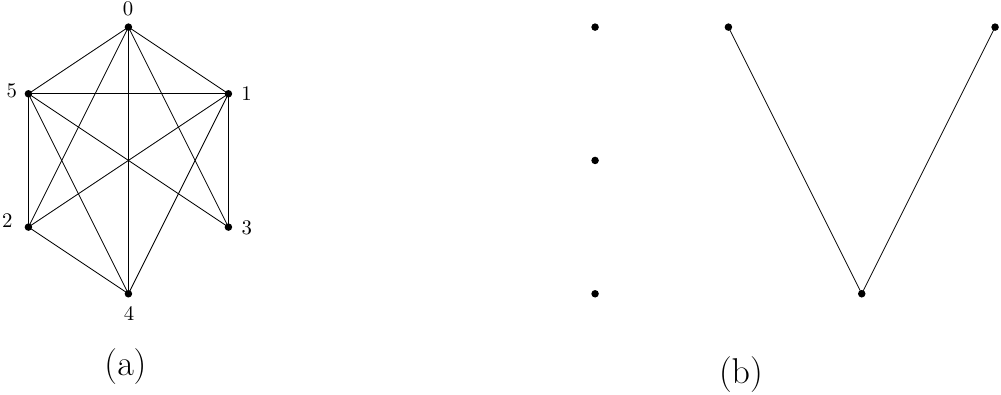}
    \caption{(a) $\mathcal{P}(\mathbb{Z}_6) $
 (b) $L(3K_2\cup P_4)$.}
    \label{fig z6}
\end{figure}
 \end{proof}
 
 \begin{proposition}{\label{complement all non-cyclic}}
     Let $G$ be a finite non-cyclic group and $\del \in \{\c , \po , \dd , \ppo \}$. Then $\del$ is the complement of a line graph of some graph $\Gamma$ if and only if $G$ is isomorphic to $Q_8$ or $\mathbb{Z}_2\times \cdots \times \mathbb{Z}_2$.
 \end{proposition}
 \begin{proof}
     Let $G$ be a finite non-cyclic group such that $\del$ is the complement of a line graph of some graph $\Gamma$. Since $G$ is non-cyclic, by Lemma \ref{M(G) >2}, we have $|\m|\geq 3$. We now discuss the following cases.\\
     \noindent\textbf{Case-1:} $|\m|\geq 4$. In this case, we show that $G\cong \mathbb{Z}_2\times \cdots \times \mathbb{Z}_2$ ($k$-copies), where $k\geq 3$. On contrary, if $G$ is not isomorphic to $\mathbb{Z}_2\times \cdots \times \mathbb{Z}_2$  then $G$ has a maximal cyclic subgroup $M$ such that $|M|\geq 3$. Consequently, $M$ has at least $2$ generators. Let $x, y \in M$ such that $M=\langle x \rangle = \langle y \rangle$ and let $z, t, w$ be generators of other three maximal cyclic subgroups of $G$. Then by Remark \ref{remark}, the subgraph induced by the set $\{x,y,z,t,w\}$ is isomorphic to $\overline{\Gamma_3}$ (see Figure \ref{fig complement_line_graph}); which is a contradiction. \\
     \noindent\textbf{Case-2:} $|\m|=3$. Consider $M_1, M_2, M_3\in \m$ such that $\phi(|M_1|)\geq \phi(|M_2|)\geq \phi(|M_3|) $. Now we have the following subcases:

     \textbf{Subcase-2.1:} $\phi(|M_1|)\geq 3$. Let $M_1=\langle x \rangle =\langle y \rangle =\langle z \rangle$  and let $M_2= \langle t \rangle $. Then the subgraph of $\del$ induced by the set $\{x,y,z,t\}$ is isomorphic to $\overline{\Gamma_1}$; a contradiction. Therefore, this subcase is not possible.

     \textbf{Subcase-2.2:} $\phi(|M_1|)\leq 2$. Then $|M_1|\in \{2,3,4,6\}$. Let $|M_1|=6$ and let $M_1= \langle x \rangle $. Then $x^2$ and $x^3$ are elements of order $3$ and $2$, respectively. Let $M_2=\langle y \rangle $. Then $M_2$ cannot contain both the elements $x^2$ and $x^3$. Otherwise, $M_1\subseteq M_2$ which is not possible. Without loss of generality, assume that $x^2 \notin M_2$. Then $x^2\nsim y$ in $\c$ and so $x^2\nsim y$ in $\del$. Consequently, the subgraph of $\del$ induced by the set $\{x, x^2, x^5, y\}$ is isomorphic to $\overline{\Gamma_1}$; again a contradiction (see Remark \ref{remark}). Thus, $|M_1|\leq 4$. Similarly, we get $|M_2|,|M_3|\leq 4$. It follows that $o(G)\leq |M_1\cup M_2\cup M_3|\leq 10$. By Table $1$ of \cite{a.kumar2022complement}, there exist only two groups $Q_8$ and $\mathbb{Z}_2\times \mathbb{Z}_2$ (whose order is at most $10$) with exactly three maximal cyclic subgroups. Thus, $G\cong Q_8$ or $G\cong \mathbb{Z}_2\times \mathbb{Z}_2$.

    Conversely, let $G\cong Q_8$. For $\del \in \{\c, \po\}$, we obtain $\del = K_2 \vee 3K_2 = \overline{L(2K_2\cup K_4)}$ (see Figure \ref{Q8}). If $\del \in \{\dd, \ppo\}$, then we have $\del = 3K_2 = \overline{L(K_4)}$. Now assume that $G \cong \mathbb{Z}_2\times \cdots \times \mathbb{Z}_2$ ($k$-times), where $k\geq 2$. For $\del \in \{\c, \po\}$, we have $\del = K_{1, 2^k-1}= \overline{L(K_2\cup K_{1, 2^k-1})}$. If $\del \in \{\dd , \ppo\}$, then note that $\del = (2^k-1)K_1= \overline{L(K_{1, 2^k-1})}$.
    \begin{figure}[ht]
    \centering
    \includegraphics[scale=.9]{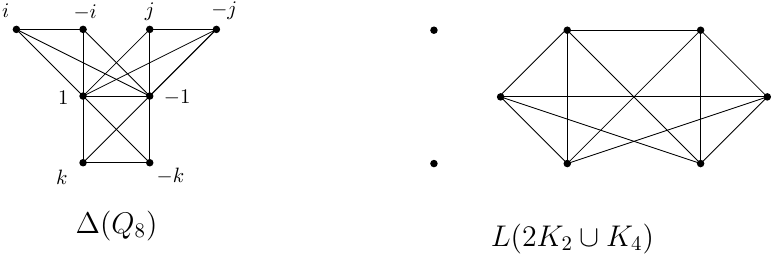}
    \caption{}
    \label{Q8}
\end{figure}
 \end{proof}
% Note that $\mathcal{P}^{**}(\mathbb{Z}_6)=\overline{L(P_4)}$, we have the following proposition for the graph $\ppo$. The proof is similar to the proof of Proposition \ref{cyclic power complement}, thus, it is omitted.
 \begin{proposition}{\label{cyclic proper complement}}
     Let $G$ be a finite cyclic group which is not a $p$-group. Then $\ppo$ is the complement of a line graph of some graph $\Gamma$ if and only if  $G\cong \mathbb{Z}_6$.
 \end{proposition}
 \begin{proof}
     First suppose that $\ppo$ is the complement of a line graph of some graph. Then in the similar lines of the proof of Proposition \ref{cyclic power complement}, we obtain $G\cong \mathbb{Z}_6$. Conversely, note that $\mathcal{P}^{**}(\mathbb{Z}_6)=\overline{L(P_4)}$. This completes our proof.
 \end{proof}
 
 Proposition \ref{cyclic power complement} together with Proposition \ref{complement all non-cyclic} yields Theorem \ref{line complement power graph}. On combining Proposition \ref{complement all non-cyclic} and Proposition \ref{cyclic proper complement}, we get Theorem \ref{line complement proper power graph}. If $G$ is a cyclic group, then $\c \cong K_n$ (cf. Theorem \ref{complete}). Observe that $K_n=\overline{L(nK_2)}$. Using these facts and Proposition \ref{complement all non-cyclic}, we obtain Theorem \ref{line complement EPG}. 
 % Theorem \ref{line complement proper EPG} yields from Proposition \ref{complement all non-cyclic}.
 % \begin{proposition}{\label{complement cyclic EPG  }}
 %    Let $G$ be a finite cyclic group. Then $\c$ is the complement of a line graph of some graph. 
 % \end{proposition}
%  \begin{proof}

%      Since $G$ is a cyclic group of order $n$, we have $\c = K_n$. We know tht $K_n=\overline{L(nK_2)}$ . This completes the proof.
% \end{proof}

 %%%%%%%%%%%%%%%%%%%%%%%%%%%%%%%%%%%%%%%%%%%%%%%%%%%%%%%%%%%%%%%%%%%%%%%%%%%%%%%%%%%%%%%%%%%%%%%%%%%%%%%%%%%%%%%%%%%%%%%%%%%%%%%%%%%%%%%%%%%%%%%%%

 %%%%%%%%%%%%%%%%%%%%%%%%%%%%%%%%%%%%%%%%%%%%%%%%%%%%%%%%%%%%%%%%%%%%%%%%%

\section*{Declarations}

\textbf{Funding}: The first author gratefully acknowledge for providing financial support to CSIR  (09/719(0110)/2019-EMR-I) government of India. The second author wishes to  acknowledge the support of Core Research Grant (CRG/2022/001142) funded by  SERB.

\vspace{.3cm}
\textbf{Conflicts of interest/Competing interests}: There is no conflict of interest regarding the publishing of this paper. 

\vspace{.3cm}
\textbf{Availability of data and material (data transparency)}: Not applicable.

\vspace{.3cm}
\textbf{Code availability (software application or custom code)}: Not applicable.

%\bibliographystyle{abbrv}
%\bibliography{References}

\begin{thebibliography}{10}

\bibitem{a.Cameron2016}
G.~Aalipour, S.~Akbari, P.~J. Cameron, R.~Nikandish, and F.~Shaveisi.
\newblock On the structure of the power graph and the enhanced power graph of a
  group.
\newblock {\em Electron. J. Combin.}, 24(3):3.16, 18, 2017.

\bibitem{a.barati2021}
Z.~Barati.
\newblock Line zero divisor graphs.
\newblock {\em J. Algebra Appl.}, 20(9):2150154, 13, 2021.

\bibitem{a.beineke1970}
L.~W. Beineke.
\newblock Characterizations of derived graphs.
\newblock {\em J. Combin. Theory}, 9:129--135, 1970.

\bibitem{a.bera2022}
S.~Bera.
\newblock Line graph characterization of power graphs of finite nilpotent
  groups.
\newblock {\em Comm. Algebra}, 50(11):4652--4668, 2022.

\bibitem{a.Bera2017}
S.~Bera and A.~K. Bhuniya.
\newblock On enhanced power graphs of finite groups.
\newblock {\em J. Algebra Appl.}, 17(8):1850146, 2018.

\bibitem{a.bera2022dominating}
S.~Bera and H.~K. Dey.
\newblock On the proper enhanced power graphs of finite nilpotent groups.
\newblock {\em J. Group Theory}, 25(6):1109--1131, 2022.

\bibitem{a.bera2021connectivity}
S.~Bera, H.~K. Dey, and S.~K. Mukherjee.
\newblock On the connectivity of enhanced power graphs of finite groups.
\newblock {\em Graphs Combin.}, 37(2):591--603, 2021.

\bibitem{a.Cameron2010}
P.~J. Cameron.
\newblock The power graph of a finite group, {II}.
\newblock {\em J. Group Theory}, 13(6):779--783, 2010.

\bibitem{a.Cameron2011}
P.~J. Cameron and S.~Ghosh.
\newblock The power graph of a finite group.
\newblock {\em Discrete Math.}, 311(13):1220--1222, 2011.

\bibitem{a.cameron2020connectivity}
P.~J. Cameron and S.~H. Jafari.
\newblock On the connectivity and independence number of power graphs of
  groups.
\newblock {\em Graphs Combin.}, 36(3):895--904, 2020.

\bibitem{a.chakrabarty2009undirected}
I.~Chakrabarty, S.~Ghosh, and M.~K. Sen.
\newblock Undirected power graphs of semigroups.
\newblock {\em Semigroup Forum}, 78(3):410--426, 2009.

\bibitem{a.chattopadhyay2021minimal}
S.~Chattopadhyay, K.~L. Patra, and B.~K. Sahoo.
\newblock Minimal cut-sets in the power graphs of certain finite non-cyclic
  groups.
\newblock {\em Comm. Algebra}, 49(3):1195--1211, 2021.

\bibitem{a.doostabadiforbidden}
A.~Doostabadi, A.~Erfanian, and M.~Farrokhi D.~G.
\newblock On power graphs of finite groups with forbidden induced subgraphs.
\newblock {\em Indag. Math. (N.S.)}, 25(3):525--533, 2014.

\bibitem{a.doostabadi2015connectivity}
A.~Doostabadi and M.~Farrokhi D.~Ghouchan.
\newblock On the connectivity of proper power graphs of finite groups.
\newblock {\em Comm. Algebra}, 43(10):4305--4319, 2015.

\bibitem{b.dummit1991abstract}
D.~S. Dummit and R.~M. Foote.
\newblock {\em Abstract algebra}.
\newblock Prentice Hall, Inc., Englewood Cliffs, NJ, 1991.

\bibitem{a.kelarev2000groups}
A.~Kelarev and S.~Quinn.
\newblock A combinatorial property and power graphs of groups.
\newblock {\em Contrib. General Algebra}, 12(58):3--6, 2000.

\bibitem{kelarev2003graph}
A.~V. Kelarev.
\newblock {\em Graph algebras and automata}, volume 257.
\newblock Marcel Dekker, Inc., New York, 2003.

\bibitem{kelarev2004labelled}
A.~V. Kelarev.
\newblock Labelled {C}ayley graphs and minimal automata.
\newblock {\em Australas. J. Combin.}, 30:95--101, 2004.

\bibitem{a.kelarev2009cayley}
A.~V. Kelarev, J.~Ryan, and J.~Yearwood.
\newblock Cayley graphs as classifiers for data mining: the influence of
  asymmetries.
\newblock {\em Discrete Math.}, 309(17):5360--5369, 2009.

\bibitem{a.powergraphsurvey}
A.~Kumar, L.~Selvaganesh, P.~J. Cameron, and T.~Tamizh~Chelvam.
\newblock Recent developments on the power graph of finite groups---a survey.
\newblock {\em AKCE Int. J. Graphs Comb.}, 18(2):65--94, 2021.

\bibitem{a.kumar2023}
J.~Kumar, X.~Ma, Parveen, and S.~Singh.
\newblock Certain properties of the enhanced power graph associated with a
  finite group.
\newblock {\em Acta Math. Hungar.}, 169(1):238--251, 2023.

\bibitem{a.masurvey2022}
X.~Ma, A.~Kelarev, Y.~Lin, and K.~Wang.
\newblock A survey on enhanced power graphs of finite groups.
\newblock {\em Electron. J. Graph Theory Appl. (EJGTA)}, 10(1):89--111, 2022.

\bibitem{a.ma2021forbidden}
X.~Ma, S.~Zahirovi\'{c}, tY.~Lv, and Y.~She.
\newblock Forbidden subgraphs in enhanced power graphs of finite groups.
\newblock {\em \rm{arXiv}:2104.04754}, 2021.

\bibitem{a.MannaForbidden2021}
P.~Manna, P.~J. Cameron, and R.~Mehatari.
\newblock Forbidden subgraphs of power graphs.
\newblock {\em Electron. J. Combin.}, 28(3):3.4, 14, 2021.

\bibitem{a.panda2021enhanced}
R.~P.~Panda, S.~Dalal, and J.~Kumar.
\newblock On the enhanced power graph of a finite group.
\newblock {\em Comm. Algebra}, 49(4):1697--1716, 2021.

\bibitem{a.kumar2022complement}
Parveen and J.~Kumar.
\newblock The complement of enhanced power graph of a finite group.
\newblock {\em arXiv:2207.04641}, 2022.

\bibitem{a.zahirovic2020study}
S.~Zahirovi\'{c}, I.~Bo\v{s}njak, and R.~Madar\'{a}sz.
\newblock A study of enhanced power graphs of finite groups.
\newblock {\em J. Algebra Appl.}, 19(4):2050062, 2020.

\end{thebibliography}

%\vspace{3 cm}
%%%%%%%%%%%%%%%%%%%%%%%%%%%%%%%%%%MATH F112-Compre-22-23 %%%%%%%%%%%%%%%%%%%%%%%%%%%%%%%%%%%%%%
% \textbf{Q-1} Let $v_1, v_2, \ldots, v_n$ be vectors in a vector space $V$ and let $T : V \rightarrow W$ be a linear transformation. Then prove or disprove the following.

% \begin{itemize}
% \item[(i)] If the set $\{T(v_1), T(v_2), \ldots, T(v_n)\}$ is linearly independent in $W$, then $\{v_1, v_2, \ldots, v_n\}$ is linearly independent in $V$. \hfill{\textbf{[3]}}
% \item[(ii)] If the set $\{v_1, v_2, \ldots, v_n\}$ is linearly independent in $V$, then $\{T(v_1), T(v_2), \ldots, T(v_n)\}$ is linearly independent in $W$. \hfill{\textbf{[3]}}
% \end{itemize}

% \textbf{Q-2} Let $T: \mathcal{P}_2 \rightarrow \mathcal{P}_2 $  be a linear transformation  such that
% \[T(1 + x) = 1+ x^2 , \; \; \;\; T(x + x^2) = x - x^2, \; \;\;\; T(1 + x^2) = 1 + x + x^2.\] 

% Find $T(4 + 3x^2)$ and $T(a + bx + cx^{2})$. \hfill{\textbf{[9]}}
% %Poole 480/17

%\textbf{Q-3} Find the real and the imaginary part of the integral $\oint_{C} \overline{z} e^{z} dz$, where $C$ is the square with vertices $0$, $1$, $1+i$, and $i$ taken with the counterclockwise direction. \hfill{\textbf{[15]}}

\vspace{1cm}
\noindent
{\bf Parveen\textsuperscript{\normalfont 1}, \bf Jitender Kumar\textsuperscript{\normalfont 1}}
\bigskip

\noindent{\bf Addresses}:

%%%%%%%%%%%%%%%%%%%%%%%%%%%%%%%%%%%%%%%%%%%%%%%%%%%%%%%%%%%%%%%%%%%%%%%%%%%%%%%%%%%%%%%%%%%%%%%%%%%%%%%

\end{document}